\documentclass[11pt]{amsart}
\usepackage{amssymb,amsmath}
\usepackage{amsthm, dsfont}
\usepackage{hyperref}
\usepackage[all]{xy}

\newcommand\comment[1]{}

\newcommand\classification[2][]{%
  \gdef\@classification{%
    \href{http://www.ams.org/msc/}%
{\textit{2000 Mathematics Subject Classification}} \ignorespaces#2\unskip}}

\leftmargin  \leftmargini
\def\GG{{\mathbb G}}
\def\NN{{\mathbb N}}
\def\ZZ{{\mathbb Z}}
\def\QQ{{\mathbb Q}}

\def\G{{\mathcal G}}
\def\H{{\mathcal H}}

\def\fH{{\mathfrak H}}
\def\fM{{\mathfrak M}}

\def\tF{\tilde{F}}
\def\tR{\tilde{R}}

\def\id{{\rm id}}

\def\ID{{\rm ID}}

\def\cat#1{{\sf #1}}

\def\vect#1{\text{\boldmath $#1$\unboldmath}} 
\def\ie{i.\,e.}

\def\isom{\cong}
\def\congr{\equiv}
\def\restr#1{|_{#1}}
\def\ideal{\unlhd}

\DeclareMathOperator{\Ker}{Ker}

\DeclareMathOperator{\Hom}{Hom}
\DeclareMathOperator{\Aut}{Aut}

\DeclareMathOperator{\GL}{GL}

\DeclareMathOperator{\Gal}{Gal}
\DeclareMathOperator{\Quot}{Quot}

\DeclareMathOperator{\Spec}{Spec}
\def\uGal{\underline{\Gal}}

\def\markdef{\bf }

\theoremstyle{plain}
\newtheorem{thm}{Theorem}[section]
\newtheorem{cor}[thm]{Corollary}

\newtheorem{prop}[thm]{Proposition}
\newtheorem{conj}[thm]{Conjecture}

\theoremstyle{definition}
\newtheorem{defn}[thm]{Definition}

\newtheorem{rem}[thm]{Remark}
\newtheorem{exam}{Example}
\newenvironment{notation}{{\bf Notation}\it}

\begin{document}

\title[Infinitesimal Galois groups]{Infinitesimal group
  schemes as iterative differential Galois groups}
\author{Andreas Maurischat}
\address{\rm {\bf Andreas Maurischat (n\'e R\"oscheisen)}, Interdisciplinary Center for
  Scientific Computing, Heidelberg 
University, Im Neuenheimer Feld 368, 69120 Heidelberg, Germany }
\email{\sf andreas.maurischat@iwr.uni-heidelberg.de}

\classification{}

\keywords{Galois theory, Differential Galois theory, inseparable
  extensions, infinitesimal group schemes}

\begin{abstract}
This article is concerned with Galois theory for iterative
differential fields (ID-fields) in positive characteristic. More
precisely, we consider purely inseparable 
Picard-Vessiot extensions, because these are the ones having an
infinitesimal group scheme as iterative differential Galois group. In
this article we prove a necessary and sufficient 
condition to decide whether an infinitesimal group scheme occurs as
Galois group scheme of a Picard-Vessiot extension over a given
ID-field or not. In particular, this solves the inverse ID-Galois
problem for infinitesimal group schemes. 
\end{abstract}

\maketitle

\section{Introduction}

In recent days, Picard-Vessiot theory for differential equations in
characteristic zero and for iterative differential equations in
positive characteristic has been extended to the case of non
algebraically closed fields of constants (cf. \cite{dyckerhoff}
resp. \cite{maurischat}). In the classical setting the Galois group of
a PV-extension is given by the points of a linear algebraic group
over the constants. In characteristic zero, one then has a Galois
correspondence between all intermediate differential fields and the
Zariski closed subgroups of the Galois group. In positive
characteristic this correspondence was restricted to intermediate
iterative differential fields over which the PV-field is
separable. This restriction in positive characteristic and similar
problems in the case of a non algebraically closed field of constants
have been removed in \cite{dyckerhoff} resp. \cite{maurischat} by
regarding the Galois group as a group scheme and not as the group of
rational points. Every intermediate (iterative) differential field is
then obtained as the field of invariants of some closed subgroup
scheme. For example an intermediate ID-field over which the PV-field
is inseparable is the field of invariants of a nonreduced subgroup
scheme. In general, a PV-extension $E/F$ can be inseparable itself and
in this case the fixed field of $E$ under 
the full group of iterative differential automorphisms of $E$ over $F$
is strictly bigger than $F$. Since classically one assumes equality,
the more general extensions are called pseudo Picard-Vessiot
extensions (PPV-extensions) here.

In this article, we treat questions concerning purely inseparable
PPV-exten\-sions. This is done in the setting of
fields with a {\em multivariate} iterative derivation and having a
perfect field of constants. (Although some of the minor results hold without
the assumption of perfectness.)
We first show
that a PPV-extension is purely
inseparable if and only if its Galois group scheme is an infinitesimal
group scheme and that the exponent of the extension and the height of
the group scheme are equal (cf. Cor. \ref{infinitesimal_group}). The
main result is a necessary and sufficient condition to decide whether
an infinitesimal group scheme occurs as Galois group of a
PPV-extension over a given ID-field or not
(cf. Thm. \ref{general-realisation} and
Cor. \ref{special-realisation}). 

In Section \ref{basics}, we introduce the reader to the basic notation of
multivariate iterative differential rings and PPV-extensions. Some
properties, general results on PPV-extensions and the Galois
correspondence are given in Section \ref{galois-theory} and can also
be found in \cite{maurischat} 
(see also \cite{heiderich}). Section \ref{purely-insep} is 
dedicated to purely inseparable PPV-extensions and the corresponding
infinitesimal group schemes. In the last section, we give some
examples to illustrate the previous results.

\medskip

{\bf Acknowledgements:} I would like to thank J.~Hartmann,
B.~H.~Matzat and M.~Wibmer for helpful comments and suggestions on the 
paper.

\section{Basic notation}\label{basics}

All rings are assumed to be commutative with unit.
We use the usual notation for
  multiindices, namely
$\binom{\vect{i}+\vect{j}}{\vect{i}}=\prod_{\mu=1}^m
\binom{i_\mu+j_\mu}{i_\mu}$ and $\vect{T}^{\vect{i}}=T_1^{i_1}
T_2^{i_2}\cdots T_m^{i_m}$ for  $\vect{i}=(i_1,\dots,
i_m),\vect{j}=(j_1,\dots, j_m)\in\NN^m$ and $\vect{T}=(T_1,\dots, T_m)$. 

An {\markdef $m$-variate iterative 
derivation} on a ring $R$ is a homomorphism of rings $\theta:R\to
R[[T_1,\dots, T_m]]$, such that $\theta^{(\vect{0})}=\id_R$ and for
all $\vect{i},\vect{j}\in \NN^m$, $\theta^{(\vect{i})}\circ \theta^{(\vect{j})}=\binom{\vect{i}+\vect{j}}{\vect{i}}\theta^{(\vect{i}+\vect{j})}$,
where the maps $\theta^{(\vect{i})}:R\to R$ are defined by
$\theta(r)=:\sum_{\vect{i}\in\NN^m}
\theta^{(\vect{i})}(r)\vect{T}^{\vect{i}}$
(cf. \cite{heiderich}, Ch. 4).
In the case $m=1$ this is equivalent to
the usual definition of an iterative derivation given for example in
\cite{mat_hart}. 
The pair $(R,\theta)$ is then called an ID-ring and
$C_R:=\{ r \in R\mid \theta(r)=r\}$ is called the {\markdef ring of
  constants} of $(R,\theta)$.\footnote{The name {\em constants} is due to the
fact that all $\theta^{(\vect{i})}$ ($\vect{i}\ne \vect{0}$) vanish at
these elements analogous to the vanishing of derivations in
characteristic zero.}
An ideal $I\ideal R$ is called an
 {\markdef ID-ideal} if $\theta(I)\subseteq I[[\vect{T}]]$ and $R$ is
 {\markdef ID-simple} if $R$ has no nontrivial
 ID-ideals. Iterative derivations are extended to localisations by
 $\theta(\frac{r}{s}):=\theta(r)\theta(s)^{-1}$ and to tensor products
 by 
$$\theta^{(\vect{k})}(r\otimes s)=\sum_{\vect{i}+\vect{j}=\vect{k}}
\theta^{(\vect{i})}(r)\otimes \theta^{(\vect{j})}(s)$$ 
for all $\vect{k}\in \NN^m$.
The $m$-variate iterative derivation $\theta$ is called
{\markdef non-dege\-ne\-ra\-te} if the $m$ additive maps
$\theta^{(1,0,\dots,0)}, \theta^{(0,1,0,\dots,0)}, \dots ,
\theta^{(0,\dots,0,1)}$ (which acutally are derivations on $R$) are
$R$-linearly independent. 

Given an ID-ring $(R,\theta_R)$ over an ID-field $(F,\theta)$, we call
an element 
$x\in R$ {\markdef differentially finite over $F$} if the $F$-vector
space spanned by all $\theta^{(\vect{k})}(x)$ ($\vect{k}\in\NN^m$)
is finite dimensional. It is easy to see that the set of elements
which are differentially finite over $F$ form an ID-subring of $R$
that contains~$F$.

\begin{rem}\label{rem-on-IDs} (see also \cite{heiderich}, Ch. 4)

Given an $m$-variate iterative derivation $\theta$ on a ring $R$, one
obtains a set of $m$ ($1$-variate) iterative derivations
$\theta_1,\dots, \theta_m$ by defining 
$$\theta_1^{(k)}:=\theta^{(k,0,\dots,0)},\quad
\theta_2^{(k)}:=\theta^{(0,k,0,\dots,0)}, \quad \dots ,\quad \theta_m^{(k)}:= 
\theta^{(0,\dots,0,k)}$$ 
for all $k\in\NN$. By the iteration rule
for $\theta$ these iterative derivations commute, i.\,e. satisfy the
condition $\theta_i^{(k)}\circ
\theta_j^{(l)}=\theta_j^{(l)}\circ\theta_i^{(k)}$ for all $i,j\in
\{1,\dots, m\}, k,l\in\NN$. On the other hand, given $m$
commuting $1$-variate iterative derivations $\theta_1,\dots, \theta_m$ one obtains
an $m$-variate iterative derivation $\theta$ by defining
$$\theta^{(\vect{k})}:=\theta_1^{(k_1)}\circ \dots \circ
\theta_m^{(k_m)}$$ for all $\vect{k}=(k_1,\dots, k_m)\in\NN^m$.

Using the iteration rule one sees that the $m$-variate iterative
derivation $\theta$ is determined by the derivations $\theta_1^{(1)},\dots,
\theta_m^{(1)}$ if the characteristic of $R$ is zero, and by the set
of maps $\{\theta_1^{(p^\ell)},\dots, \theta_m^{(p^\ell)}\mid
\ell\in\NN\}$ if the characteristic of $R$ is $p>0$. Furthermore,
$\theta$ is non-degenerate if and only if for all $j=1,\dots,
m$ the derivation $\theta_j^{(1)}$ is nontrivial on $\bigcap_{i=1}^{j-1}
\Ker(\theta_i^{(1)})$.

\smallskip

Next we consider the case that $R=:F$ is a field of positive characteristic
$p$ and that $\theta$ is non-degenerate. Then the derivations
$\theta_1^{(1)},\dots, \theta_m^{(1)}$ are nilpotent $C_F$-endomorphisms of
$F$. Since they commute and $\theta$ is non-degenerate, there exist
$x_1,\dots, x_m\in F$ such that $\theta_i^{(1)}(x_j)=\delta_{ij}$ for
all $i,j$, where $\delta_{ij}$ denotes the Kronecker delta.
Therefore $\{ x_1^{e_1}\cdots x_m^{e_m} \mid 0\leq e_j\leq p-1 \}$
is a basis of $F$ as a vector space over $F_1:= \bigcap_{i=1}^{m}
\Ker(\theta_i^{(1)})$. Hence $F/F_1$ is a field extension of degree $p^m$.
Furthermore, the maps $\theta_1^{(p)},\dots, \theta_m^{(p)}$ are
derivations on $F_1$, they also are nilpotent and commute, and
$$\theta_i^{(p)}(x_j^p)=\left(\theta_i^{(1)}(x_j)\right)^p=\delta_{ij}.$$
So by the same argument, $F_1$ is a vector space over
$F_2:=F_1\cap \bigcap_{i=1}^{m} \Ker(\theta_i^{(p)})$ and
$[F_1:F_2]=p^m$. Repeating this, one obtains a descending sequence of
subfields $F_\ell:= F_{\ell-1}\cap \bigcap_{i=1}^{m}
\Ker(\theta_i^{(p^{\ell-1})})$ satisfying $[F_{\ell-1}:F_\ell]=p^m$.

This sequence will be useful in Section \ref{purely-insep}.
\end{rem}

\begin{defn}
Let $(F,\theta)$ be an ID-field, and
let $A=\sum_{\vect{k}\in \NN^m} A_{\vect{k}} \vect{T}^{\vect{k}}\in
\GL_n(F[[\vect{T}]])$ be a matrix satisfying the properties
$A_{\vect{0}}=\mathds{1}_n$ and 
$\binom{\vect{k}+\vect{l}}{\vect{l}}A_{\vect{k}+\vect{l}}=\sum_{\vect{i}+\vect{j}=\vect{l}}
\theta^{(\vect{i})}(A_{\vect{k}}) A_{\vect{j}}$ for all $\vect{k},\vect{l}\in\NN^m$.
Then an equation
$$\theta(\vect{y})=A\vect{y},$$
where $\vect{y}$ is a vector of indeterminants, is called an {\markdef
  iterative differential equation} (IDE) over $F$.\footnote{Throughout
  this article, iterative derivations are applied componentwise to
  vectors and matrices.}
\end{defn}



 \begin{defn}
An ID-ring $(R,\theta_R)\geq (F,\theta)$ is called a {\markdef pseudo
  Picard-Vessiot ring} (PPV-ring) for $\theta(\vect{y})=A\vect{y}$ if the
  following holds: 
\begin{enumerate}
\item $R$ is an ID-simple ring.
\item There is a fundamental solution matrix $Y\in\GL_n(R)$, \ie{} an invertible
  matrix satisfying $\theta(Y)=AY$.
\item As an $F$-algebra, $R$ is generated by the coefficients of $Y$
  and $\det(Y)^{-1}$.
\item $C_R=C_F$.
\end{enumerate}
The quotient field $E=\Quot(R)$ (which exists, since such a PPV-ring
is always an integral domain) is called a {\markdef pseudo
  Picard-Vessiot field} (PPV-field)
for the IDE $\theta(\vect{y})=A\vect{y}$.
\end{defn}

\begin{rem}
The condition on the $A_{\vect{k}}$ given in the definition of the IDE
is equivalent to the condition that
$\theta_R^{(\vect{k})}(\theta_R^{(\vect{l})}(Y_{ij}))=
\binom{\vect{k}+\vect{l}}{\vect{k}}\theta_R^{(\vect{k}+\vect{l})}(Y_{ij})$ holds for a fundamental
solution matrix $Y=(Y_{ij})_{1\leq i,j\leq n}\in\GL_n(R)$.

Furthermore, the condition $A_{\vect{0}}=\mathds{1}_n$ already implies
that the matrix $A$ is invertible.
\end{rem}

\begin{notation}
From now on, $(F,\theta)$ denotes an ID-field of positive
characteristic $p$, and $K=C_F$ its field of
constants. We assume that $K$ is perfect, and that the $m$-variate
iterative derivation $\theta$ is non-degenerate.
\end{notation}

\section{Galois theory}\label{galois-theory}

In this section, we deal with the Galois group scheme corresponding to a
PPV-extension. We will see various facettes of the group structure and
group action, and provide the Galois correspondence for
PPV-extensions.

We begin with a characterisation of the PPV-ring in a PPV-field.

\begin{prop}\label{diff-finite}
Let $(R,\theta_R)$ be a PPV-ring over $F$ for an IDE
$\theta(\vect{y})=A\vect{y}$ and $E=\Quot(R)$. Then $R$ is equal to
the set of elements in $E$ which are differentially finite over $F$.
\end{prop}

\begin{proof} (Compare \cite{mat_hart}, Thm. 4.9, for the case
  when  $K$ is
  algebraically closed and $\theta$ is univariate%
.)\\
Let $Y\in\GL_n(R)$ be a fundamental solution matrix for the IDE. Then
by definition $\theta^{(\vect{k})}(Y)=A_{\vect{k}}Y$ and hence for all
$i,j$ and all $\vect{k}\in\NN^m$ the derivatives
$\theta^{(\vect{k})}(Y_{ij})$ are in the
$F$-vector space spanned by all $Y_{ij}$, i.\,e. all $Y_{ij}$ are
differentially finite. Furthermore, one has 
$\theta(\det(Y)^{-1})=\det(\theta(Y))^{-1}=\det(AY)^{-1}=\det(A)^{-1}\det(Y)^{-1}$,
i.\,e. $\det(Y)^{-1}$ is differentially finite. Therefore, $R$ is
generated by differentially finite elements, and since the 
differentially finite elements form a ring, all elements of $R$ are
differentially finite.

On the other hand, let $x\in E$ be differentially finite over
$F$ and let $W_F(x)$ be the $F$-vector
space spanned by all $\theta^{(\vect{k})}(x)$ ($\vect{k}\in\NN^m$). 
Then the set $I_x:=\{ r\in R\mid r\cdot W_F(x)\subseteq R\}$ is
an ID-ideal of $R$. Since $W_F(x)$ is finite dimensional and $E$ is
the quotient field of $R$, one has $I_x\ne 0$. Since $R$ is ID-simple,
this implies $I_x=R$. Hence $1\cdot W_F(x)\subseteq R$, and in
particular $1\cdot x=x\in R$.
\end{proof}

From this characterisation of the PPV-ring as the ring of differentially finite
elements, we immediately get the following.

\begin{cor}\label{unique-PPV-ring}
Let $E$ be a PPV-field over $F$ for several IDEs. Then the PPV-ring
inside $E$ is unique and independent of the particular IDE.
\end{cor}

\subsection{The Galois group scheme}

For a PPV-ring $R/F$ we define the functor
$$\underline{\Aut}^{\ID}(R/F): (\cat{Algebras} / K) \to (\cat{Groups}),
L\mapsto \Aut^{\ID}(R\otimes_K L/F\otimes_K L)$$
where $L$ is provided with the trivial iterative derivation.

In \cite{maurischat}, Sect. 10, it is shown that the functor
$\G:=\underline{\Aut}^{\ID}(R/F)$ is represent\-able by a $K$-algebra of
finite type and hence is an affine group scheme of finite type over 
$K$, which is called the (iterative differential) {\markdef Galois
  group scheme} of the extension $R$ over $F$ -- denoted by
$\underline{\Gal}(R/F)$ --, or also the
Galois group scheme of  the extension $E$ over $F$, $\underline{\Gal}(E/F)$, where
$E=\Quot(R)$ is the corresponding PPV-field.\footnote{This is justified by the
fact given in Corollary \ref{unique-PPV-ring} that the PPV-ring can be
recovered from the PPV-field without regarding an IDE.
Also take care that the functor $\underline{\Aut}^{\ID}(E/F)$ is not
isomorphic to $\underline{\Aut}^{\ID}(R/F)$. Hence the Galois group
scheme of $E/F$ has to be defined using the PPV-ring.}
Furthermore $\Spec(R)$ is a $(\G \times_{K}F)$-torsor and the
corresponding isomorphism of rings
$$\gamma:R\otimes_F R\to R\otimes_K K[\G]$$
is an $R$-linear ID-isomorphism.

By restricting $\gamma$ to the constants, one obtains that
$K[\G]$ is isomorphic 
to $C_{R\otimes_F R}$. One checks by calculation (see also \cite{takeuchi})
that the comultiplication on $K[\G]$ is induced via
this isomorphism by the map
$$R\otimes_F R\longrightarrow (R\otimes_F R)\otimes_R (R\otimes_F R),
a\otimes b\mapsto (a\otimes 1)\otimes (1\otimes b),$$
and the counit map $ev:K[\G]\to K$ is induced by the
multiplication 
$$R\otimes_F R\longrightarrow R, a\otimes b\mapsto ab.$$

\comment{
More generally, we have
\begin{prop}\label{central_iso}
Let $R$ be a PPV-ring for an IDE $\theta(\vect{y})=A\vect{y}$ with
fundamental solution matrix $\bar{X}\in\GL_n(R)$. Let
$T\geq F$ be an ID-simple 
ID-ring with $C_T=K$ such that there exists a fundamental solution
matrix $Y\in \GL_n(T)$. Furthermore, let $U:=C_{T\otimes_F R}$ be the
$K$-algebra of constants of $T\otimes_F R$.
Then there exists a $T$-linear ID-isomorphism
$$\gamma_T: T\otimes_F R \to T\otimes_K U,$$
which is given by $\bar{X}_{ij}\mapsto \sum_{k=1}^n Y_{ik}\otimes
\bar{Z}_{kj}$ for some elements $\bar{Z}_{kj}\in U$.
\end{prop}

\begin{thm}
Let $(T,\theta_T)$ be an ID-simple ID-ring over $F$ with $C_T=C_F$ and 
having a fundamental solution matrix $Y\in\GL_n(T)$ for some IDE
$\theta(\vect{y})=A\vect{y}$. Then the subalgebra of $T$ generated by
the coefficients of $Y$ and by $\det(Y)^{-1}$ is the unique PPV-ring
inside $T$ for this IDE.
\end{thm}

\begin{proof}
Since for two fundamental solution matrices $Y$ and $\tilde{Y}$ in
$T$, the coefficients of $Y^{-1}\tilde{Y}$ are constants,
i.\,e. elements in $C_T=C_F$, the subalgebra of $T$ generated by the
coefficients and the inverse of the determinant is the same for every
fundamental solution matrix. This proves the uniqueness of a PPV-ring
inside $T$.

For showing that this subalgebra is a PPV-ring, we only have to show
that it is ID-simple.

\vdots

\end{proof}

} 

\comment{
\subsection{Galois correspondence}\label{galois_corres}

Let $S$ be a $K$-algebra and $\H/K$ be a subgroup functor of the
functor $\underline{\Aut}(S/K)$, \ie{} for every $K$-algebra $L$, the set
$\H(L)$ is a group acting on $S_L$ and this action is functorial in $L$. An
element $s\in S$ is then called {\markdef invariant} if for all $L$,
the element $s\otimes 1\in S_L$ is invariant under $\H(L)$. 
The ring of invariants is denoted by $S^{\H}$. (In \cite{jantzen},
I.2.10 the invariant elements are called ``fixed points''.)

Let $E=\Quot(S)$ be the localisation of $S$ by all non zero divisors.
Since every automorphism of $S\otimes_K L$ extends uniquely to an
automorphism of $\Quot(S\otimes_K L)$, the functor
$\underline{\Aut}(S/K)$ is a subgroup functor of the group functor
$$(\cat{Algebras} / K) \to (\cat{Groups}),
L\mapsto \Aut(\Quot(S\otimes_K L)/\Quot(F\otimes_K L)).$$
In this sense, we call an element $e=\frac{r}{s}\in E$ {\markdef
  invariant} under $\H$, if for all 
$K$-algebras $L$ and all $h\in\H(L)$,
$$\frac{h.(r\otimes 1)}{h.(s\otimes 1)}=\frac{r\otimes 1}{s\otimes
  1}=e\otimes 1.$$
The ring of invariants of $E$ is denoted by $E^{\H}$.

In the following, let again $R$ be a PPV-ring over $F$, $E=\Quot(R)$ its
field of fractions and $\G=\underline{\Gal}(R/F)$ the Galois group
scheme of $R$ over $F$. Furthermore, let
$\rho:=\gamma\restr{1\otimes R}:R\to R\otimes K[\G]$ be the
ID-homomorphism which describes the action of $\G$ on $R$.
} 

Let $\H \leq \G$ be a subgroup functor, \ie{} for every
$K$-algebra $L$, the set $\H(L)$ is a group acting on $R_L:=R\otimes_K
L$ and this action is functorial in $L$. 
An element $r\in R$ is then called {\markdef invariant} under $\H$ if
for all $L$, the element $r\otimes 1\in R_L$ is invariant under
$\H(L)$. The ring of invariants is denoted by $R^{\H}$. (In
\cite{jantzen}, I.2.10 the invariant elements are called ``fixed
points''.)

Let $E=\Quot(R)$ be the quotient field and for all $L$ let
$\Quot(R\otimes_K L)$ be the localisation by all nonzero divisors.
Since every automorphism of $R\otimes_K L$ extends uniquely to an
automorphism of $\Quot(R\otimes_K L)$, the functor
$\underline{\Aut}(R/F)$ is a subgroup functor of the group functor
$$(\cat{Algebras} / K) \to (\cat{Groups}),
L\mapsto \Aut(\Quot(R\otimes_K L)/\Quot(F\otimes_K L)).$$
In this sense, we call an element $e=\frac{r}{s}\in E$ {\markdef
  invariant} under $\H$, if for all 
$K$-algebras $L$ and all $h\in\H(L)$,
$$\frac{h.(r\otimes 1)}{h.(s\otimes 1)}=\frac{r\otimes 1}{s\otimes
  1}=e\otimes 1.$$
The ring of invariants of $E$ is denoted by $E^{\H}$.

\begin{rem}
The action of $\G:=\underline{\Gal}(R/F)$ on $R$ is fully described by the
ID-homomorphism $\rho:=\gamma\restr{1\otimes R}:R\to R\otimes_K K[\G]$.
Namely, for a $K$-algebra $L$ and $g\in \G(L)\isom \Hom(K[\G],L)$, one
has $g.(r\otimes 1)=(1\otimes g)(\rho(r))\in R\otimes_K L$ for all
$r\in R$.
\end{rem}

\begin{prop}
Let $E/F$ be a PPV-extension with PPV-ring $R$ and Galois group scheme $\G$.
An ID-field $\tF$, with $F\leq \tF\leq E$, is a PPV-field over $F$, if
and only if it is stable under the action of $\G$, i.\,e. if $\rho(R\cap
\tF)\subseteq (R\cap \tF)\otimes K[\G]$.
\end{prop}

\begin{proof}
If $\tF$ is a PPV-field, its PPV-ring $\tR$ is the set of elements in
$\tF$ which are
differentially finite over $F$ (cf. Prop \ref{diff-finite}), in
particular we have $\tR=\tF\cap R$.
Hence we obtain a commutative diagram:
\begin{center}$\xymatrix{
\tR\otimes_F \tR \ar[r]^-{\isom} \ar[d] & \tR\otimes_K K[\uGal(\tR/F)]=\tR\otimes_K
C_{\tR\otimes_F \tR} \ar[d] \\
R\otimes_F R \ar[r]^-{\isom}& R\otimes_K K[\G]= R\otimes_K C_{R\otimes_F R}
}$\end{center}
But this implies $\rho(\tR)\subseteq \tR\otimes_K
C_{\tR\otimes_F \tR}\subseteq \tR\otimes_K K[\G]$, i.\,e. $\tF$ is
stable under the action of $\G$.

The converse is given in Theorem \ref{galois_correspondence},iii).
\end{proof}

\begin{thm}{\bf (Galois correspondence)}\label{galois_correspondence}
Let $E/F$ be a PPV-extension with PPV-ring $R$ and Galois group scheme
$\G$.
\begin{enumerate}
\item There is an antiisomorphism of the lattices
$$\fH:=\{ \H \mid \H\leq\G \text{ closed subgroup scheme of }\G
\}$$
and
$$\fM:=\{ M \mid F\leq M\leq E \text{ intermediate ID-field} \}$$
given by 
$\Psi:\fH \to \fM,\H\mapsto E^{\H}$ and 
$\Phi:\fM \to \fH, M\mapsto \underline{\Gal}(E/M)$.
\item\label{normal_subgroup} If $\H\leq \G$ is normal, then $E^{\H}=\Quot(R^{\H})$ and
  $R^{\H}$ is a PPV-ring over $F$ with Galois group scheme 
$\underline{\Gal}(R^{\H}/F)\isom \G/\H$.
\item If $M\in\fM$ is stable under the action of $\G$, then $\H:=\Phi(M)$
 is a normal subgroup scheme of $\G$, $M$ is a PPV-extension of $F$ and
$\underline{\Gal}(M/F)\isom \G/\H$.
\item For $\H\in \fH$, the extension $E/E^{\H}$ is separable if and
  only if $\H$ is reduced.
\end{enumerate}
\end{thm}

\begin{proof}
See \cite{maurischat}, Thm. 11.5.
\end{proof}

For a purely inseparable field extension $E/F$ one denotes by ${\rm
  e}(E/F)$ the {\markdef exponent} of the extension, i.\,e. the minimal number
$e\in\NN$ such that $E^{p^e}\subseteq F$. For an infinitesimal group
scheme $\G$ over $K$, the {\markdef height} of $\G$, denoted by ${\rm h}(\G)$, is the
minimal number $h\in \NN$ such that $x^{p^h}=0$ for all $x\in
K[\G]^+$. (Here $K[\G]^+$ is the kernel of the counit map
$ev:K[\G]\to K$ and is a nilpotent ideal by the definition of an
infinitesimal group scheme.)

\begin{cor}\label{infinitesimal_group}
Let $E/F$ be a PPV-extension with Galois group scheme $\G$.
Then $E/F$ is a purely inseparable extension if and only if
$\G$ is an infinitesimal group scheme. In this case, the exponent ${\rm
  e}(E/F)$ and the height ${\rm h}(\G)$ are equal.
\end{cor}

\begin{proof}
Let $\G$ be infinitesimal of height $h$ and let $ev:K[\G]\to K$ denote the
evaluation map corresponding to the neutral element of the group. Then 
for any $\frac{r}{s}\in E$, we have
$(\id\otimes ev)(\gamma(r\otimes s-s\otimes r))=0$,
i.\,e. $\gamma(r\otimes s-s\otimes r)\in R\otimes_K K[\G]^+$.
Since $\G$ is of height $h$, we obtain
$(r\otimes s-s\otimes r)^{p^h}=0$. Therefore
$r^{p^h}\otimes s^{p^h}=s^{p^h}\otimes r^{p^h} \in R\otimes_F R$
which means that $\frac{r^{p^h}}{s^{p^h}}\in F$. So $E/F$ is purely
inseparable of exponent $\leq h$.
On the other hand, let $E/F$ be purely inseparable of exponent $e$.
For arbitrary $x\in K[\G]^+$, let $\gamma^{-1}(1\otimes x)=:\sum_j
r_j\otimes s_j$. Then 
$$1\otimes x^{p^e}
= \gamma\left(\sum_j r_j^{p^e} \otimes s_j^{p^e}\right)
= \gamma\left(\sum_j r_j^{p^e}s_j^{p^e}\otimes 1\right)
= \sum_j r_j^{p^e}s_j^{p^e}\otimes 1.$$
Hence (e.g. by applying $\id \otimes ev$), one obtains $\sum_j
r_j^{p^e}s_j^{p^e}=0$ and $x^{p^e}=0$. Therefore $\G$ is infinitesimal
of height $\leq e$.
\end{proof}

\section{Purely inseparable extensions}\label{purely-insep}

As in the previous section, $F$ denotes a field of positive
characteristic $p$ with a non-degenerate $m$-variate iterative
derivation $\theta$ and a perfect field of constants $K=C_F$.

\begin{notation}
For all $\ell\in\NN$, 
let $J_\ell:=\left\{ (j_1,\dots,j_m)\in\NN^m\setminus
\{\vect{0}\}\mid \forall\, i: j_i<p^\ell \right\}$ and let
$$F_\ell:=\bigcap_{\vect{j}\in J_\ell} \Ker(\theta_F^{(\vect{j})}).$$
Actually, the subfields $F_\ell$ are the same as the ones defined in
Remark \ref{rem-on-IDs}.

Since $\theta_F(F_\ell)\subseteq F_\ell[[T_1^{p^\ell},\dots, T_m^{p^\ell}]]$, one obtains an
iterative derivation on $F_{[\ell]}:=(F_\ell)^{p^{-\ell}}$ by
$\theta_{F_{[\ell]}}(x):=\left(\theta_F(x^{p^\ell})\right)^{p^{-\ell}}$.
Obviously, it is the unique iterative derivation which turns
$F_{[\ell]}$ into an ID-extension of $F$.
\end{notation}

\begin{prop}\label{max-id-extension}

\begin{enumerate}
\item\label{leq-ell} For all $\ell\in\NN$, $F_{[\ell]}$ is the unique maximal purely
inseparable ID-extension of $F$ of exponent $\leq \ell$.
\item\label{formula} For all $\ell_1,\ell_2\in\NN$,
  $(F_{[\ell_1]})_{[\ell_2]}=F_{[\ell_1+\ell_2]}$.
\item\label{trivial} If $F_{[1]}=F$ then $F_{[\ell]}=F$ for all $\ell\in\NN$.
\item\label{eq-ell} If $F_{[1]}\ne F$ and $\theta$ is non-degenerate, then for all
  $\ell\in\NN$, the exponent of $F_{[\ell]}/F$ is exactly $\ell$.
\end{enumerate}
\end{prop}

\begin{proof}
For the proof of part \ref{leq-ell}, we have already seen that
$F_{[\ell]}/F$ is an ID-extension, and by definition it is purely
inseparable of exponent $\leq \ell$. If $E$ is 
a purely inseparable ID-extension of $F$ of exponent $\leq \ell$, then
$E^{p^\ell}\subseteq F \cap E_{\ell}\subseteq F_{\ell}$ and therefore
$E\subseteq F_{[\ell]}$. Hence $F_{[\ell]}$ is the unique maximal
ID-extension of this kind.

By definition $(F_{[\ell_1]})_{[\ell_2]}$ is an ID-extension of $F$ of
exponent $\leq \ell_1+\ell_2$. Hence by part \ref{leq-ell}, we have
$(F_{[\ell_1]})_{[\ell_2]}\subseteq F_{[\ell_1+\ell_2]}$. On the other
hand $\left(F_{[\ell_1+\ell_2]}\right)^{p^{\ell_1+\ell_2}}\subseteq F$
and so $\left(F_{[\ell_1+\ell_2]}\right)^{p^{\ell_2}}\subseteq
F_{[\ell_1]}$. Hence $F_{[\ell_1+\ell_2]}$ is an ID-extension of
$F_{[\ell_1]}$ of exponent $\leq \ell_2$ and therefore contained in
$(F_{[\ell_1]})_{[\ell_2]}$. This proves part \ref{formula}.

Part \ref{trivial} is a direct consequence of part \ref{formula}. So
it remains to prove \ref{eq-ell}. For this it suffices to show that
$F_{[\ell+1]}\ne F_{[\ell]}$ for all $\ell$, because this implies that
${\rm e}(F_{[\ell]}/F)\geq {\rm e}(F_{[\ell-1]}/F)+1\geq \dots \geq 
{\rm e}(F_{[1]}/F)+\ell-1=\ell$.

By Remark \ref{rem-on-IDs}, one has $\dim_{F_{\ell+1}}(F_\ell)=p^m$,
since $\theta$ is non-degenerate. Assume that
$F_{[\ell+1]}=F_{[\ell]}$. Then
$F_{\ell+1}=\left(F_{[\ell+1]}\right)^{p^{\ell+1}}
=\left(F_{[\ell]}\right)^{p^{\ell+1}}=(F_\ell)^p$
and therefore $F$ is a finite extension of $(F_\ell)^p$ of degree
$[F:(F_\ell)^p]=[F:F_{\ell+1}]=p^{(\ell+1) m}$. On the other hand,
$$[F:(F_\ell)^p]=[F:F^p]\cdot [F^p:(F_\ell)^p]=[F:F^p]\cdot
[F:F_\ell]=p^{\ell m}[F:F^p].$$
So $[F:F^p]=p^m=[F:F_1]$, and hence $F_1=F^p$, in contradiction to
$F_{[1]}\ne F$.
\end{proof}

\begin{thm}
Let $E/F$ be a PPV-extension and let $\ell\in\NN$. Then
$E_{[\ell]}/F_{[\ell]}$ is a PPV-extension, and its Galois group
scheme is related to $\uGal(E/F)$ by $({\bf
  Frob}^\ell)^{*}\left(\uGal(E_{[\ell]}/F_{[\ell]})\right)\isom
\uGal(E/F)$, where ${\bf Frob}$ denotes the Frobenius
morphism on $\Spec(K)$.
\end{thm}

\begin{proof}
Let $R\subseteq E$ be the corresponding PPV-ring and $Y\in\GL_n(R)$ a
fundamental solution matrix for  a corresponding IDE
$\theta(\vect{y})=A\vect{y}$. Since the $m$-variate iterative
derivation is non-degenerate on $F$, on has
$[F:F_\ell]=p^{m\ell}=[E:E_\ell]$.
Hence, there is a matrix $D \in \GL_n(F)$ such that
$\tilde{Y}:=D^{-1}Y\in \GL_n(R_\ell)$. The matrix $\tilde{Y}$
satisfies
$$\theta(\tilde{Y})=\theta(D^{-1}Y)=\theta(D)^{-1}AD\tilde{Y},$$
i.\,e. it is a fundamental solution matrix for the IDE
$\theta(\vect{y})=\tilde{A}\vect{y}$, where
$\tilde{A}=\theta(D)^{-1}AD\in\GL_n(F[[\vect{T}]])$.

We first show that $\tilde{A}\in\GL_n(F_\ell[[T_1^{p^\ell},\dots, T_m^{p^\ell}]])$:
Clearly $\tilde{A}\in\GL_n(F[[\vect{T}^{p^\ell}]])$, since
$\theta^{(\vect{k})}(\tilde{Y})=0$ for all
$\vect{k}\in J_\ell$ and since $\theta$ is
iterative.
Then for all $\vect{j}\in\NN^m$ and all $\vect{k}\in J_\ell$ we have
$$\theta^{(\vect{k})}\left(\theta^{(\vect{j})}(\tilde{Y})\right)=
\theta^{(\vect{j})}\left(\theta^{(\vect{k})}(\tilde{Y})\right)=0,$$
and
$$\theta^{(\vect{k})}\left(\theta^{(\vect{j})}(\tilde{Y})\right)=\theta^{(\vect{k})}\left(\tilde{A}_{\vect{j}}\cdot
\tilde{Y}\right)=
\theta^{(\vect{k})}(\tilde{A}_{\vect{j}})\tilde{Y}.$$
Hence, $\theta^{(\vect{k})}(\tilde{A}_{\vect{j}})=0$. Therefore
$\tilde{A}_{\vect{j}}$ has coefficients in $F_\ell$.

Since $\tilde{A}\in\GL_n(F_\ell[[\vect{T}^{p^\ell}]])$,
$R_\ell$ is actually a PPV-ring over $F_\ell$ with fundamental solution matrix $\tilde{Y}$.

\comment{\bf !! erwaehnen weshalb $R_\ell$ ID-simple?}

By taking $p^{\ell}$-th roots, we obtain that $R_{[\ell]}$ is a
PPV-ring over $F_{[\ell]}$ with fundamental solution matrix $\left((\tilde{Y}_{i,j})^{p^{-\ell}}\right)_{i,j}$.

For obtaining the relation between the Galois groups, we first observe that
$F$ and $R_\ell$ are linearly disjoint over $F_\ell$ and hence
$F\otimes_{F_\ell} R_\ell\isom R$, 
which induces a natural isomorphism of the Galois groups
$\uGal(R/F)\isom \uGal(R_\ell/F_\ell)$.

Furthermore the $p^\ell$-th power Frobenius endomorphism leads to an
isomorphism
$$R_{[\ell]}\otimes_{F_{[\ell]}} R_{[\ell]} \xrightarrow{()^{p^\ell}}
R_\ell\otimes_{F_\ell} R_\ell.$$
Since $\uGal(R_\ell/F_\ell)$ (resp.$\uGal(R_{[\ell]}/F_{[\ell]})$ is isomorphic as $K$-group scheme to
$\Spec(C_{R_\ell\otimes_{F_\ell} R_\ell})$ (resp.
$\Spec(C_{R_{[\ell]}\otimes_{F_{[\ell]}} R_{[\ell]}})$), this gives
the desired property 
$$({\bf Frob}^\ell)^{*}\left(\uGal(E_{[\ell]}/F_{[\ell]})\right)\isom \uGal(E_\ell/F_\ell)
\isom\uGal(E/F).$$
\end{proof}

From this theorem we obtain a criterion for $E_{[\ell]}/E$ being a
PPV-extension.

\begin{cor}\label{E_ell-is-ppv}
Let $E/F$ be a PPV-extension and suppose that $F_1=F^p$. Then the
extension $E_{[\ell]}/E$ is a PPV-extension, for all $\ell\in\NN$.
\end{cor}

\begin{proof}
From $F_1=F^p$, it follows that $F_{[\ell]}=F$ for all $\ell$. Hence
by the previous theorem, $E_{[\ell]}/F$ is a PPV-extension and
therefore $E_{[\ell]}/E$ is a PPV-extension.
\end{proof}

\begin{prop}\label{finite-id-ext}
Let $E$ be a finite ID-extension of some ID-field $F$ with $C_E=K$.
Then there is a finite field extension $L$ over $K$ such that $E$ is
contained in a PPV-extension of $FL=F\otimes_K L$.  
\end{prop}

\begin{proof}
Let $e_1,\dots, e_n\in E$ be an $F$-basis of $E$. Then there are unique
$A_{\vect{k}}\in F^{n\times n}$, such that $\theta_E^{(\vect{k})}(e_i)=\sum_{j=1}^n
(A_{\vect{k}})_{ij}e_j$ for all $\vect{k}\in\NN^m$ and $i=1,\dots, n$.
Since the $A_{\vect{k}}$ are unique, the property of $\theta_E$ being an
iterative derivation implies that $\theta(\vect{y})=A\vect{y}$ is an
iterative differential equation, where $A=\sum_{\vect{k}\in\NN^m}
A_{\vect{k}} \vect{T}^{\vect{k}}\in \GL_n(F[[\vect{T}]])$.
Let $U:=E[X_{ij},\det(X)^{-1}]$ be the universal solution ring for
this IDE over $E$ (i.\,e. $\theta_U(X)=A X$). Then the ideal
$(x_{11}-e_1,x_{21}-e_2,\dots, x_{n1}-e_n)\ideal U$ is an ID-ideal and
there is a maximal ID-ideal $P$ containing $(x_{11}-e_1,\dots,
x_{n1}-e_n)$. Then the field of constants $L:=C_{U/P}$ of $U/P$ is a
finite field extension of $K$ and by construction $U/P$ is a
PPV-extension of $FL$ which contains $E$.
\end{proof}

\begin{thm}\label{general-realisation}
Let $F$ be an ID-field with $C_F=K$ perfect.\\
Let $\tilde{C}_{\ell}$ denote  the
maximal subalgebra of $C_{F_{[\ell]}\otimes_F F_{[\ell]}}$ which is a
Hopf algebra with respect to the comultiplication induced by 
$$F_{[\ell]}\otimes_F F_{[\ell]}\longrightarrow
\left(F_{[\ell]}\otimes_F F_{[\ell]}\right)\otimes_{F_{[\ell]}}\left(
F_{[\ell]}\otimes_F F_{[\ell]}\right), a\otimes b\mapsto (a\otimes
1)\otimes (1\otimes b).$$

Then an infinitesimal group scheme of height $\leq \ell$ is realisable as
ID-Galois group scheme over $F$, if and only if it is a factor group
of $\Spec(\tilde{C}_{\ell})$.
\end{thm}

\begin{proof}
Let $\tilde{\G}$ be an infinitesimal group scheme of height $\leq \ell$
which is realisable as Galois group scheme over $F$ and let $F'/F$ be
an extension with Galois group scheme $\tilde{\G}$. By Cor.
\ref{infinitesimal_group} and Prop. \ref{max-id-extension}, $F'$ is an
ID-subfield of $F_{[\ell]}$. Therefore,
$K[\tilde{\G}]\isom C_{F'\otimes_F F'}$ is a subalgebra of $C_{F_{[\ell]}\otimes_F
  F_{[\ell]}}$ and is a Hopf algebra with comultiplication as
given in the statement. Hence it is a sub-Hopf algebra of
$\tilde{C}_{\ell}$ and so $\tilde{\G}$ is a factor group of $\Spec(\tilde{C}_{\ell})$.

For the converse,
we first assume that there is a PPV-extension $E/F$ such that
$E\supseteq F_{[\ell]}$. Let $R$ denote the corresponding PPV-ring and
$\G:=\uGal(E/F)$ the Galois group scheme. Since $F_{[\ell]}$ is an
intermediate ID-field, there is a subgroup $\H\leq \G$ such that
$F_{[\ell]}=E^{\H}$. Since all elements in $F_{[\ell]}$ are
differentially finite over $F$ we even have $F_{[\ell]}=R^{\H}$.
Then $\tilde{C}_{\ell}\subseteq C_{F_{[\ell]}\otimes_F
  F_{[\ell]}}\subseteq C_{R\otimes_F R}\isom K[\G]$ is a sub-Hopf
algebra, i.\,e. $\Spec(\tilde{C}_{\ell})$ is a factor group of $\G$.

If $\tilde{\G}$ is a factor group of $\Spec(\tilde{C}_{\ell})$ then it
is a factor group of $\G$ and therefore there is a normal subgroup
$\G'\ideal \G$ such that $\tilde{\G}\isom \G/\G'$. Then by the
Galois correspondence, $\tilde{F}:=E^{\G'}$ is a PPV-extension of $F$ with
Galois group scheme $\tilde{\G}$.

If there is no PPV-extension $E/F$ containing $F_{[\ell]}$, then by
Prop. \ref{finite-id-ext}, there is a finite Galois extension $K'$ of
$K$ such that there 
is a PPV-extension $E'/FK'$ containing $F_{[\ell]}K'$. By the
previous arguments there is a PPV-field $F'$ over $FK'$ with Galois
group $\tilde{\G}\times_K K'$. Since $F'$ is a purely inseparable
extension of $FK'$, it is defined over $F$, i.\,e. there is an
ID-field $\tilde{F}/F$ such that $F'=\tF\otimes_K K'$. 
Since $\Gal(K'/K)$ acts on $F'=\tF K'$ by ID-automorphisms, the
constants of $\tF\otimes_F \tF\isom (F'\otimes_F
\tF)^{\Gal(K'/K)}\isom (F'\otimes_{FK'} F')^{\Gal(K'/K)}$ are equal to
the $\Gal(K'/K)$-invariants of $C_{F'\otimes_{FK'} F'}\isom
K'[\tilde{\G}]$ inside $C_{F_{[\ell]}\otimes_F
  F_{[\ell]}}K'$, \ie{} are equal to $K[\tilde{\G}]$. By comparing
dimensions, one obtains that the $\tF$-linear mapping $\tF\otimes_K
K[\tilde{\G}]\to \tF\otimes_F \tF$ is in fact an isomorphism, and
hence by \cite{maurischat}, Prop. 10.12, $\tF/F$ is a PPV-extension with
Galois group scheme $\tilde{\G}$.
\end{proof}

\begin{cor}\label{special-realisation}
Let $E$ be an ID-field and suppose that $E$ is a PPV-extension of
some ID-field $F$ satisfying $F_1=F^p$.
An infinitesimal group scheme of height $\leq \ell$ is realisable as
ID-Galois group scheme over $E$, if and only if it is a factor group
of $\uGal(E_{[\ell]}/E)$.
\end{cor}

\begin{proof}
This follows directly from Theorem \ref{general-realisation} and the
fact that in this case $E_{[\ell]}/E$ is a PPV-extension by Corollary
\ref{E_ell-is-ppv}.
\end{proof}

\section{Examples}

In this section we consider some examples. Troughout this section $K$
denotes a perfect field of characteristic $p>0$ and $K((t))$ is
equipped with the univariate iterative derivation $\theta$ given by
$\theta(t)=t+T$.

\begin{exam}
We start with the easiest case. If $F=K(t)$ or $F$ is a finite
ID-extension of $K(t)$ inside $K((t))$, then $F_1=F^p$,
i.\,e. $F_{[1]}=F$, and therefore by Prop. \ref{max-id-extension},
there exist no purely inseparable ID-extensions of $F$.
For $F=K(t)$, the property $F_1=F^p$ is obvious,  and for $F$ being a finite
extension of $K(t)$, it is obtained by a simple dimension argument.
\end{exam}

\begin{exam}
We present an example for an ID-field $F$ with $F_{[\ell]}\gneq F$
which nevertheless has no purely inseparable PPV-extensions. More
precisely, we show that the constants of $F_{[\ell]}\otimes_F
F_{[\ell]}$ are equal to $K=C_F$ for all $\ell\in\NN$.

\smallskip

Let $\alpha\in \ZZ_p\setminus \QQ$ be a $p$-adic integer, and for all
$k\in \NN$, let $\alpha_k\in \{0,\dots, p^k-1\}$ be chosen such that
$\alpha\congr \alpha_k \mod{p^k}$. Then we define
$r:=\sum_{k=1}^\infty t^{\alpha_k} \in K[[t]]$.
The field $F:=K(t,r)$ is then an ID-subfield of $K((t))$, since for
all $j\in\NN$, 
\begin{eqnarray*}
\theta^{(p^j)}(r)&=& \sum_{k=1}^\infty
\theta^{(p^j)}\left(t^{\alpha_k}\right)
= \sum_{k=1}^\infty \binom{\alpha_k}{p^j} t^{\alpha_k-p^j}\\
&=& \binom{\alpha_{j+1}}{p^j}  t^{-p^j}\sum_{k=j+1}^\infty t^{\alpha_k}
= \binom{\alpha_{j+1}}{p^j}  t^{-p^j}\left( r- \sum_{k=1}^j
t^{\alpha_k}\right)\in K(t,r).
\end{eqnarray*}
Here we used that $\binom{a}{p^j}=0$ if $a<p^j$ and
$\binom{a}{p^j}\congr \binom{b}{p^j} \mod{p}$ if $a\congr b \mod{p^{j+1}}$.

\smallskip

We will show now that $r$ is transcendental over $K(t)$:

Let $s$ be a solution for the $1$-dimensional IDE
$\theta^{(p^j)}(y)=\binom{\alpha_{j+1}}{p^j}t^{-p^j}y$ ($j\in\NN$) in
some extension field of $F$. Since $\alpha\not\in \QQ$, the element
$s$ is transcendental over $K(t)$ by \cite{mat_hart}, Thm. 3.13.
One then easily verifies 
$$\theta^{(p^j)}\begin{pmatrix} s& r\\ 0&1\end{pmatrix}=
\begin{pmatrix} \binom{\alpha_{j+1}}{p^j}t^{-p^j} & -\binom{\alpha_{j+1}}{p^j} \sum_{k=1}^j
  t^{\alpha_k-p^j} \\ 0&0\end{pmatrix} \cdot \begin{pmatrix} s&
    r\\ 0&1\end{pmatrix},$$
which shows that $K(t,r,s)$ is a PPV-field over $K(t)$ with Galois
group inside $\GG_m\ltimes \GG_a\isom \{\left(\begin{smallmatrix}
  x&a\\ 0&1\end{smallmatrix}\right)\in \GL_2\}$.

Since $s$ is transcendental over $K(t)$, the full subgroup $\GG_m$ is
contained in the Galois group. The only subgroups of $\GG_a$ which are
stable under the $\GG_m$-action are the Frobenius kernels
$\alpha_{p^m}$. But all Galois groups over $K(t)$ are reduced
(cf. \cite{maurischat}, Cor. 11.7), and
hence we have $\uGal(K(t,r,s)/K(t))=\GG_m\ltimes \GG_a$ or
$=\GG_m$. In both cases $K(t,r,s)$ contains no elements that are
algebraic over $K(t)$. Since the power series of $r$ does not become
eventually periodic, $r\not\in K(t)$ and so $r$ has to be
transcendental over $K(t)$.

Next we are going to calculate the constants of $F_{[\ell]}\otimes_F
F_{[\ell]}$:

It is easily seen that $F_{[\ell]}=K(t,r_{[\ell]})$, where
$$r_{[\ell]}:= \left( t^{-\alpha_\ell}(r-\sum_{k=1}^\ell
t^{\alpha_k})\right)^{p^{-\ell}}= \sum_{k=1}^\infty
t^{(\alpha_{k+\ell}-\alpha_\ell)p^{-\ell}}\in K[[t]],$$
and the derivatives of $r_{[\ell]}$ are given by:
$$\theta^{(p^j)}(r_{[\ell]})=
\binom{(\alpha_{j+1+\ell}-\alpha_{\ell})p^{-\ell}}{p^j}
t^{-p^j}\left( r_{[\ell]} - \sum_{k=1}^j
t^{(\alpha_{k+\ell}-\alpha_\ell)p^{-\ell}} \right).$$
Hence, one obtains for all $n\in\NN$: $$\theta^{(n)}(r_{[\ell]})\in
\binom{(\alpha-\alpha_{\ell})p^{-\ell}}{n} t^{-n}r_{[\ell]}
+ K(t).$$

For calculating the constants in $F_{[\ell]}\otimes_F F_{[\ell]}$, we
remark that $\{ r_{[\ell]}^i\otimes r_{[\ell]}^j \mid 0\leq i,j\leq
p^\ell-1\}$ is a basis of $F_{[\ell]}\otimes_F F_{[\ell]}$ as an
$F$-vector space. A further calculation shows that for $n\in\NN$ and
$k\in\ZZ$ 
$$\theta^{(n)}\left( t^k r_{[\ell]}^i\otimes r_{[\ell]}^j\right)
\congr
\binom{k+ (i+j)(\alpha-\alpha_{\ell})p^{-\ell}}{n}t^{-n}\left(
t^k r_{[\ell]}^i\otimes r_{[\ell]}^j\right)$$
modulo terms in $r_{[\ell]}^{\mu}\otimes r_{[\ell]}^{\nu}$ with $\mu+\nu<i+j$.
So an element $x:=\sum_{i,j} c_{i,j}r_{[\ell]}^i\otimes r_{[\ell]}^j\in
F_{[\ell]}\otimes_F F_{[\ell]}$ can only be constant, if for the terms
of maximal degree these binomial coefficients vanish for all $n$.
Since $\alpha$ is not rational, this is only possible if $i=j=0$ is
the maximal degree and if $k=0$, \ie{} $x\in K$.
So we have shown that $C_{F_{[\ell]}\otimes_F F_{[\ell]}}=K$ for all
$\ell\in\NN$, which implies by Theorem \ref{general-realisation} that
there are no purely inseparable PPV-extensions over $F=K(t,r)$.
\end{exam}

\begin{exam}
The following example is quite contrary to the previous one. In this example all
purely inseparable ID-extensions are PPV-extensions.

Let $\alpha_1,\dots, \alpha_n\in\ZZ_p$ be $p$-adic integers such that
the set $\{1,\alpha_1,\dots, \alpha_n\}$ is $\ZZ$-linear independent,
and let $\alpha_i=:\sum_{k=0}^\infty a_{i,k}p^k$ ($i=1,\dots, n$) be
their normal series, \ie{} $a_{i,k}\in\{ 0,\dots, p-1\}$. For
$i=1,\dots, n$, we then
define
$$s_i:=\sum_{k=0}^\infty a_{i,k}t^{p^k}\in K((t))$$
and consider the field $F:=K(t,s_1,\dots, s_n)$ which obviously is an
ID-subfield of~$K((t))$.
Since $\theta^{(p^\ell)}(s_i)=a_{i,\ell}$ for all $\ell\in\NN$ and
$i=1,\dots, n$, the extension $F/K(t)$ is a PPV-extension and its
Galois group scheme is a subgroup scheme of~$\GG_a^{\,n}$. Actually, the
condition on the $\alpha_i$ implies that the $s_i$ are algebraically
independent over $K(t)$ and hence the Galois group scheme is the full
group~$\GG_a^{\,n}$. Therefore by Corollary \ref{E_ell-is-ppv}, for all
$\ell\in\NN$ the extension $F_{[\ell]}/F$ is a PPV-extension and
$\uGal(F_{[\ell]}/F)\isom (\boldsymbol{\alpha}_{p^\ell})^n$, where
$\boldsymbol{\alpha}_{p^\ell}$ denotes the kernel of the $p^\ell$-th
power Frobenius map on~$\GG_a^{\,n}$.
Furthermore, $(\boldsymbol{\alpha}_{p^\ell})^n$ is a commutative group
scheme and so all its subgroup schemes are normal subgroup
schemes. By Theorem \ref{galois_correspondence}, this implies that
every intermediate ID-field $F\leq E\leq  F_{[\ell]}$ is a
PPV-extension of $F$. So all purely inseparable ID-extensions of $F$
are PPV-extensions over~$F$.
Furthermore, by Cor. \ref{special-realisation}, an infinitesimal
group scheme is realisable over $F$ if and only if it is a closed subgroup
scheme of $(\boldsymbol{\alpha}_{p^\ell})^n$ for some $\ell$, \ie{} if
and only if it is a closed infinitesimal subgroup scheme of~$\GG_a^{\,n}$.
\end{exam}

\comment{
\section{Some Conjectures}

\begin{conj}
Prop. \ref{finite-id-ext} with $C_F$ perfect.
\end{conj}

\begin{conj}
Let $(F,\theta)$ be an ID-field with a perfect field of constants
$C:=C_F$. Assume further that $trdeg(F/C)<\infty$.
Then there exist inseparable PPV-extensions of $F$ if and
only if $F_{[1]}\ne F$.
\end{conj}

\begin{proof}
The only if part has already been proved in \cite{maurischat}. So we
have to show the if-part. Let $E:=F_{[1]}$. By
assumption, $E$ is a purely inseparable ID-extension of $F$. 
By Prop. \ref{finite-id-ext}, there is a PPV-extension $\tilde{E}$
over $F':=FC^{\rm alg}$ containing $EC^{\rm alg}$.
\dots maybe:
By Galois decent (see [Oes09]), we obtain that
$E':=\tilde{E}^{\Gal(C^{\rm alg}/C)}$ is a PPV-extension of $F$. So
$E'$ is in fact an inseparable PPV-extension of $F$. 
\end{proof}

} 


\vspace*{.5cm}

\parindent0cm

\end{document}